\documentclass[12pt]{amsart}
\usepackage{epsfig}
\usepackage{amsmath,amssymb,amsbsy,amsfonts,latexsym,
amsopn,amstext,amsopn,amstext,amsxtra,euscript,amscd}

\itemsep=\smallskipamount

\newfont{\teneufm}{eufm10}
\newfont{\seveneufm}{eufm7}
\newfont{\fiveeufm}{eufm5}
\newfam\eufmfam
  \textfont\eufmfam=\teneufm \scriptfont\eufmfam=\seveneufm
  \scriptscriptfont\eufmfam=\fiveeufm
\def\cB{{\mathcal B}}

\def\ep{\mathbf{e}_p}

\def\mand{\qquad \text{and} \qquad}

\def\bbbr{{\rm I\!R}} %reelle Zahlen

\def\bbbn{{\rm I\!N}} %natuerliche Zahlen

\def\bbbc{{\mathchoice {\setbox0=\hbox{$\displaystyle\rm C$}\hbox{\hbox
to0pt{\kern0.4\wd0\vrule height0.9\ht0\hss}\box0}}
{\setbox0=\hbox{$\textstyle\rm C$}\hbox{\hbox
to0pt{\kern0.4\wd0\vrule height0.9\ht0\hss}\box0}}
{\setbox0=\hbox{$\scriptstyle\rm C$}\hbox{\hbox
to0pt{\kern0.4\wd0\vrule height0.9\ht0\hss}\box0}}
{\setbox0=\hbox{$\scriptscriptstyle\rm C$}\hbox{\hbox
to0pt{\kern0.4\wd0\vrule height0.9\ht0\hss}\box0}}}}
\def\bbbq{{\mathchoice {\setbox0=\hbox{$\displaystyle\rm
Q$}\hbox{\raise
0.15\ht0\hbox to0pt{\kern0.4\wd0\vrule height0.8\ht0\hss}\box0}}
{\setbox0=\hbox{$\textstyle\rm Q$}\hbox{\raise
0.15\ht0\hbox to0pt{\kern0.4\wd0\vrule height0.8\ht0\hss}\box0}}
{\setbox0=\hbox{$\scriptstyle\rm Q$}\hbox{\raise
0.15\ht0\hbox to0pt{\kern0.4\wd0\vrule height0.7\ht0\hss}\box0}}
{\setbox0=\hbox{$\scriptscriptstyle\rm Q$}\hbox{\raise
0.15\ht0\hbox to0pt{\kern0.4\wd0\vrule height0.7\ht0\hss}\box0}}}}
\def\bbbt{{\mathchoice {\setbox0=\hbox{$\displaystyle\rm
T$}\hbox{\hbox to0pt{\kern0.3\wd0\vrule height0.9\ht0\hss}\box0}}
{\setbox0=\hbox{$\textstyle\rm T$}\hbox{\hbox
to0pt{\kern0.3\wd0\vrule height0.9\ht0\hss}\box0}}
{\setbox0=\hbox{$\scriptstyle\rm T$}\hbox{\hbox
to0pt{\kern0.3\wd0\vrule height0.9\ht0\hss}\box0}}
{\setbox0=\hbox{$\scriptscriptstyle\rm T$}\hbox{\hbox
to0pt{\kern0.3\wd0\vrule height0.9\ht0\hss}\box0}}}}
\def\bbbs{{\mathchoice
{\setbox0=\hbox{$\displaystyle     \rm S$}\hbox{\raise0.5\ht0\hbox
to0pt{\kern0.35\wd0\vrule height0.45\ht0\hss}\hbox
to0pt{\kern0.55\wd0\vrule height0.5\ht0\hss}\box0}}
{\setbox0=\hbox{$\textstyle        \rm S$}\hbox{\raise0.5\ht0\hbox
to0pt{\kern0.35\wd0\vrule height0.45\ht0\hss}\hbox
to0pt{\kern0.55\wd0\vrule height0.5\ht0\hss}\box0}}
{\setbox0=\hbox{$\scriptstyle      \rm S$}\hbox{\raise0.5\ht0\hbox
to0pt{\kern0.35\wd0\vrule height0.45\ht0\hss}\raise0.05\ht0\hbox
to0pt{\kern0.5\wd0\vrule height0.45\ht0\hss}\box0}}
{\setbox0=\hbox{$\scriptscriptstyle\rm S$}\hbox{\raise0.5\ht0\hbox
to0pt{\kern0.4\wd0\vrule height0.45\ht0\hss}\raise0.05\ht0\hbox
to0pt{\kern0.55\wd0\vrule height0.45\ht0\hss}\box0}}}}
\def\bbbz{{\mathchoice {\hbox{$\sf\textstyle Z\kern-0.4em Z$}}
{\hbox{$\sf\textstyle Z\kern-0.4em Z$}}
{\hbox{$\sf\scriptstyle Z\kern-0.3em Z$}}
{\hbox{$\sf\scriptscriptstyle Z\kern-0.2em Z$}}}}

\newtheorem{theorem}{Theorem}
\newtheorem{lemma}[theorem]{Lemma}

\def\vec#1{\mathbf{#1}}

\def\squareforqed{\hbox{\rlap{$\sqcap$}$\sqcup$}}
\def\qed{\ifmmode\squareforqed\else{\unskip\nobreak\hfil
\penalty50\hskip1em\null\nobreak\hfil\squareforqed
\parfillskip=0pt\finalhyphendemerits=0\endgraf}\fi}

\newcommand{\ignore}[1]{}

\def \Z{{\bbbz}}

\def \N{{\bbbn}}

\def \R{{\bbbr}}

\def\\{\cr}
\def\({\left(}
\def\){\right)}
\def\fl#1{\left\lfloor#1\right\rfloor}
\def\rf#1{\left\lceil#1\right\rceil}

\def\ep{\mbox{\bf{e}}_p}

\def\brho{\boldsymbol\rho} 
\def\blambda{\boldsymbol\lambda} 
\def \be{\vec{e}}

\begin{document}

\title[Exponential and Character Sums with Monomials]{Multiple Exponential and Character Sums with Monomials}

 \author[I. E. Shparlinski] {Igor E. Shparlinski}

\address{Department of Pure Mathematics, University of New South Wales,
Sydney, NSW 2052, Australia}
\email{igor.shparlinski@unsw.edu.au}

\begin{abstract} We obtain new bounds of multivariate exponential sums
with monomials, when the variables run over rather short intervals. 
Furthermore, we use the same method to derive estimates on similar 
sums with multiplicative characters to which previously known methods 
do not apply. In particular, in the multiplicative characters modulo a prime $p$
we break the barrier of $p^{1/4}$ for ranges of 
individual variables. 
\end{abstract}



%\paragraph*{Mathematical Subject Classification (2010):}\quad  05C81, 05C85, 11K38, 11L07, 68R10

 \maketitle

\section{Background}

Let $p$ be a prime number and let $n \ge 2$ be a fixed 
integer.

We also assume that we are given a system   $\brho= (\rho_1, \ldots, \rho_n)$ 
of $n$ complex-valued weights with
\begin{equation}
\label{eq:rho}
|\rho_j(x)| \le 1, \  x \in \R,\qquad j =1, \ldots, n,
\end{equation}
and an integer vector $\be$ with non-zero components:
\begin{equation}
\label{eq:ej}
\be  = (e_1, \ldots, e_{n}) \in \Z^{n}\mand e_1\ldots e_n \ne 0.
\end{equation}

We define the multiple exponential sums 
with monomials
$$
S_{p}(\lambda,\brho,\be;\cB) = \sum_{\substack{(x_1, \ldots, x_n) \in \cB\\ 
x_1\ldots x_n\not \equiv 0 \pmod p}} 
\rho_1(x_1) \ldots  \rho_n(x_n)
\ep(\lambda x_1^{e_1}\ldots x_n^{e_n}),$$
where $\lambda \in \Z$, over an $n$-dimensional cube
\begin{equation}
\label{eq:Box}
\cB = [k_1+1, k_1+h]\times  \ldots \times[k_n+1, k_n+h]
\end{equation}
with a side length $h<p$ and some integers $k_1,   \ldots, k_n$, 
where $\ep(z) = \exp(2 \pi z/p)$ and for a negative exponent $e$ 
the inversion in $x^e$ is taken modulo $p$.

The case when
\begin{equation}
\label{eq:multilin Kloost}
\rho_j(x) = \exp(2 \pi i \lambda_j x/p) \mand e_j = -1, \qquad j =1, \ldots, n, 
\end{equation}
for some integers  $\lambda_1, \ldots, \lambda_n \in \Z$, 
corresponds to multivariate incomplete Kloosterman sums. In this 
case we denote these sums as  
$$K_p(\blambda; \cB) = \sum_{\substack{(x_1, \ldots, x_n) \in \cB\\ 
x_1\ldots x_n\not \equiv 0 \pmod p}} 
\ep(\lambda x_1^{-1}\ldots x_n^{-1} + \lambda_1 x_1 + \ldots + \lambda_n x_n )
$$ 
where $\blambda = (\lambda, \lambda_1, \ldots, \lambda_n)$.
 
Using the Deligne bound~\cite{Del},  Luo~\cite{Luo}  
has shown that for~\eqref{eq:multilin Kloost}
we have
$$
K_{p}(\blambda;\cB) = O(p^{n-1} + p^{n/2}(\log p)^n), \qquad \gcd(\lambda, p)=1. 
$$
Furthermore, using the Burgess bound of character sums (see~\cite[Theorem~12.6]{IwKow}), 
Luo~\cite[Theorem~1]{Luo} has also given a bound on $K_{p}(\blambda;\cB)$ which is nontrivial 
when $h\ge p^{1/4+\varepsilon}$, provided that $n$ is large enough (see also~\cite[Theorem~2]{Luo} 
that applies to the sums $S_{p}(\lambda,\brho,\be;\cB)$).
This bound (but not the range of $h$) has been improved
in~\cite{Shp} by bringing in an additional argument based on a result of 
Ayyad,  Cochrane and  Zheng~\cite[Theorem~2]{ACZ}. 
Unfortunately, the  improvement claimed in~\cite{WaLi}  has never been justified
(as the proof of~\cite[Lemma~1]{WaLi} is invalid). Note that the method of~\cite{WaLi}
may still work for the cube at the origin, that is, for $k_1 = \ldots = k_n = 0$. 
Bourgain and Garaev~\cite{BouGar2} have recently obtained a series of 
estimates of Kloosterman sums over much smaller boxes, for example, for boxes with
the side length $h \ge p^{c/n^2}$ for some absolute constant $c>0$. 

Finally, we note that the sums $S_{p}(\lambda,\brho,\be;\cB)$ can be treated 
as general multilinear sums using the results of Bourgain~\cite[Theorems~3 and~5]{Bour1} 
and~\cite[Theorem~3]{Bour2} and of Garaev~\cite[Theorem~4.1]{Gar}. 
For example, if $\rho_j(x) = 1$ for $x \in \R$, $j =1, \ldots, n$,
then~\cite[Theorem~4.1]{Gar} implies a nontrivial estimate 
$$
 \sum_{\substack{(x_1, \ldots, x_n) \in \cB\\ 
x_1\ldots x_n\not \equiv 0 \pmod p}} 
\ep(\lambda x_1^{e_1}\ldots x_n^{e_n}) =  O\(h^n p^{-\delta}\)
$$
if $h > p^{81/(n+160)+\varepsilon}$
for some fixed $\varepsilon > 0$, where $\delta>0$ depends on $\varepsilon$ and $n$.
It is quite possible that this bound can be extended to the 
sums $K_{p}(\blambda;\cB)$ and even to $S_{p}(\lambda,\brho,\be;\cB)$. 

Here we show that the results of~\cite{BGKS3} can 
be used to get better  and fully explicit estimates for the sums $S_{p}(\lambda,\brho,\be;\cB)$
for almost all primes $p$. Moreover, for $4 \le n \le 7$, we obtain new 
bounds that  hold for all primes $p$. These bounds enable us to estimate sums
over a large number of variables, however unfortunately this larger number 
of variables does not bring any additional gains, see, however, a discussion in 
Section~\ref{sec:com}.

Using the same approach, we also derive similar bounds for the sums
$$
T_{p}(\lambda, \chi, \brho,\be;\cB) = \sum_{\substack{(x_1, \ldots, x_n) \in \cB\\ 
x_1\ldots x_n\not \equiv 0 \pmod p}} 
\rho_1(x_1) \ldots  \rho_n(x_n)
\chi(x_1^{e_1}\ldots x_n^{e_n}+\lambda),
$$
with a  multiplicative character $\chi$ modulo $p$ and $\lambda \in \Z$. 
Note that the methods of~\cite{Bour1,Bour2,Gar}
do not seem to apply to such sums.

It is interesting to 
note that in many cases our results go beyond the restriction 
$h \ge p^{1/4}$, which is associated with the Burgess bound
(see~\cite[Theorem~12.6]{IwKow}). 
For exponential sums $S_{p}(\lambda,\brho,\be;\cB)$, in the case of 
positive exponents $e_1, \ldots, e_n$, one can apply general bounds for 
exponential sums with polynomials obtained via the method of 
Vinogradov, see~\cite{Wool}, which allows to study rather short sums
with general polynomials. 
Furthermore, even shorter exponential sums (with positive 
and negative exponents $e_1, \ldots, e_n$)
can be estimated via the methods of additive combinatorics, see~\cite{Bour1,Bour2,Gar}.
However, for the sums of multiplicative characters, these approaches do not apply.
Thus this work seems to be 
the first example where the Burgess 
barrier of $p^{1/4}$ for the range of individual variables 
in  multiplicative character sums has been broken,
see~\cite{BouCha}.

Throughout, any implied constants in the symbols $O$ and
$\ll$ may occasionally depend, where obvious, on the integer $n$, the vector $\be$, 
and the integer parameters $r$ and $\nu$, 
but are absolute otherwise. We recall that the
notations $A \ll B$ and  $A = O(B)$ are both equivalent to the
statement that the inequality $|A| \le c\,B$ holds with some
constant $c> 0$.

\section{Congruences with products and character sums}
\label{sec:cong}

For a prime $p$ and integers $h\ge3$, $\nu
\ge 1$ and $k$, we
denote by $I_{p,\nu}(h,k)$ the number of solutions of the
congruence
\begin{equation*}
\begin{split}
(x_1+k)\ldots(x_{\nu}+k)\equiv & (y_1+k)\ldots(y_{\nu}+k)\not\equiv0 \pmod p,\\
1\le x_j,y_j & \le   h_j,  \qquad j =1, \ldots, \nu.
\end{split}
\end{equation*}

We now define  $d_2 = 2$ and 
$$
d_\nu = \max\{\nu^2-2\nu-2,\nu^2-3\nu+4\}, \qquad 
\nu = 3, 4, \ldots.
$$
We need the following estimate which for $\nu=2$ is
a special case of~\cite[Theorem~1]{ACZ} and for $\nu \ge 3$
follows from~\cite[Theorem~17]{BGKS2}.

\begin{lemma}
\label{lem:GProdSet}
 Let $\nu \ge 2$ be a fixed integer.
 Then for any integers $k$ and $h<p$ we have 
 $$
I_{p,\nu}(h,k) \le  \(h^{\nu} +  h^{2\nu} p^{-\nu/d_\nu} \) h^{o(1)}. 
$$
\end{lemma}

We also need the following estimate from~\cite{BGKS3}:

\begin{lemma}
\label{lem:GProdSet AlmostAll} Let $\nu \ge 1$ be a fixed integer.
Then for a sufficiently large positive integers  $T \ge h\ge3$,
for all but $o(T/\log T)$ primes $p \le T$  
and any integers $k$ and $h<p$, we have the bound
$$
I_{p,\nu}(h,k) \le \(h^\nu + h^{2\nu-1/2}p^{-1/2}\)h^{o(1)}. 
$$
\end{lemma}

Furthermore, sometimes we need to estimate the number of solutions
to   a more general congruence. For a prime $p$, an integer $\nu \ge 1$,
and  vectors
\begin{equation*}
\begin{split}
\vec{h} & = (h_1, \ldots, h_{\nu})\in \N^{\nu}, \\
\vec{k} &= (k_1, \ldots, k_{\nu}) \in \Z^{\nu}, \\
\vec{e} & = (e_1, \ldots, e_{\nu}) \in Z^{\nu},
\end{split}
\end{equation*}
we denote by
$J_{p,\nu}(\vec{e}, \vec{h},\vec{k})$ the number of solutions of the congruence
\begin{equation*}
\begin{split}
(x_1+k_1)^{e_1}\ldots   (x_{\nu}+k_{\nu})^{e_{\nu}}&\equiv  (y_1+k_1)^{e_1}\ldots  (y_{\nu}+k_{\nu})^{e_{\nu}}
\not \equiv 0 \pmod p,\\
1\le x_j,y_j & \le   h_j,  \qquad j =1, \ldots, \nu.
\end{split}
\end{equation*}

The following result, which is a slight generalisation of a similar statement from~\cite{BGKS2}, 
relates $J_{p,\nu}(\vec{e}, \vec{h},\vec{s})$ and
$I_{p,\nu}(h,k_j)$, $j =1, \ldots, \nu$.  Here we always have 
$h_1 = \ldots = h_\nu = h$, however we record this simple inequality 
in full generality.

\begin{lemma}
\label{lem:Kss} We have
$$
J_{p,\nu}(\vec{e}, \vec{h},\vec{k})\le \prod_{j=1}^{\nu}  I_{p,\nu}(h_j,k_j)^{1/\nu}.
$$
\end{lemma}

\begin{proof}
Using the orthogonality of multiplicative characters, we write
\begin{equation*}
\begin{split}
J_{p,\nu}(\vec{e}, \vec{h},\vec{s})  = \frac{1}{p-1} \sum_{\substack{1\le
x_1,y_1\le h_1\\ x_1, y_1 \not \equiv
-k_1 \pmod p}}   \ldots &
\sum_{\substack{1\le
x_{\nu}, y_\nu \le h_{ \nu}\\ x_{\nu}, y_\nu \not \equiv
-k_{\nu} \pmod p}}\\
& \sum_{\chi} \chi\(\prod_{j=1}^\nu \frac{(x_j+k_j)^{e_j}}{(y_j+k_j)^{e_j}}\),
\end{split}
\end{equation*}
where $\chi$ runs through all multiplicative characters modulo $p$.

Since $\chi(z^{-1}) = \overline\chi(z)$, we obtain 
\begin{equation}
\label{eq:K prod}
\begin{split}
J_{p,\nu}(\vec{e}, \vec{h},\vec{s})  
 = \frac{1}{p-1} \sum_{\chi} & \prod_{j=1}^{\nu}
\left|\sum_{1\le
x_j\le h_j} \chi^{e_j}\(x_j+k_j\)\right|^2,  
\end{split}
\end{equation}
(we also define $\chi(0) = \chi^{-1}(0) = 0$).
Using the H{\"o}lder inequality, we obtain
$$
J_{p,\nu}(\vec{e}, \vec{h},\vec{s})^{\nu}  \le \frac{1}{p-1}     \prod_{j=1}^{\nu}
\sum_{\chi} \left|\sum_{1\le x_j\le h_j} \chi^{e_j}\(x_j+k_j\)\right|^{2\nu}. 
$$
For any integer $e\ne 0$, when $\chi$ runs 
runs through all multiplicative characters modulo $p$, 
the character $\chi^e$ takes the same value no more than 
$\gcd(e,p-1) \le |e|$ times. 
Therefore 
\begin{equation}
\label{eq:K Hold}
J_{p,\nu}(\vec{e}, \vec{h},\vec{s})^{\nu}  \le \frac{1}{p-1}     \prod_{j=1}^{\nu} |e_j|
\sum_{\chi} \left|\sum_{1\le x_j\le h_j} \chi\(x_j+k_j\)\right|^{2\nu}. 
\end{equation}

Similarly to~\eqref{eq:K prod}, we also obtain 
$$
I_{p,\nu}(h,k) =  \frac{1}{p-1} \sum_{\chi}  \left|\sum_{1\le
x \le h} \chi\(x+k\)\right|^{2\nu}, 
$$
which together with~\eqref{eq:K Hold} implies the desired result.
\end{proof}

We also need the following well-known result which is slight generalisation
of the classical result of Davenport and Erd{\H o}s~\cite{DavErd}, which 
in turn follows from the Weil bound of multiplicative character sums,
see~\cite[Theorem~11.23]{IwKow}.

\begin{lemma}
\label{lem:DavErd} For any non-trivial
multiplicative character $\chi$ modulo $p$, integers $k$ and $h<p$,  any
 complex-valued weights with
$$
|\rho(x)| \le 1, \quad  x \in [k+1,k+h],
$$ 
and integer $\lambda$ with $\gcd(\lambda,p)=1$,  we have
$$
\sum_{u=1}^{p-1} \left|\sum_{x=k+1}^{k+h}\rho(x) \chi(ux+\lambda) \right|^{2r}
\ll  
 \left\{  \begin{array}{ll}
h p & \quad \text{if}\  r=1, \\
h^r p +  h^{2r}p^{1/2} & \quad \text{if}\  r=2,3,\ldots. \\
\end{array} \right.
$$
\end{lemma}

\section{Bounds for  all primes}

We now derive a nontrivial bound on the sums $|S_{p}(\lambda,\brho,\be;\cB)|$
for $4 \le n \le 7$. The proof is based on Lemma~\ref{lem:GProdSet},
which unfortunately is not strong enough to produce a nontrivial 
result for $n \ge 8$.  

\begin{theorem}
\label{thm:Kloost}  
For  any  prime  $p$,   integer  $h \le p$, cube $\cB$ of the form~\eqref{eq:Box},
weights   $\brho$ with~\eqref{eq:rho}, vector  $\be$  with~\eqref{eq:ej}
and integer $\lambda$ with $\gcd(\lambda,p)=1$, 
we have
$$
|S_{p}(\lambda,\brho,\be;\cB)| \le  h^{o(1)}\cdot
 \left\{  \begin{array}{ll}
h^{2} p^{1/2} +     h^{4}p^{-1/2} & \quad \text{if}\  n=4, \\
h^{5/2} p^{1/2} +     h^{4}p^{1/8} + h^{5}p^{-3/8} & \quad \text{if}\  n=5, \\
h^{3} p^{1/2} +     h^{6}p^{-1/4}& \quad \text{if}\ n=6,\\
  h^{11/2}p^{1/4} + h^{7}p^{-1/8}
& \quad \text{if}\ n=7.
\end{array} \right.
$$
\end{theorem}

\begin{proof} Let $s = \fl{n/2}$ and $t = n-s = \rf{n/2}$.
We define 
\begin{equation*}
\begin{split}
\cB_1 & =  [k_1+1, k_1+h]\times \ldots \times[k_s+1, k_s+h], \\
\cB_2 & =  [k_{s+1}+1, k_{s+1}+h]\times  \ldots \times [k_n+1, k_n+h], 
\end{split}
\end{equation*}
and write 
\begin{equation*}
\begin{split}
\eta_1(u) & =  \sum_{\substack{(x_1, \ldots, x_s) \in \cB_1\\ x_1^{e_1}\ldots x_s^{e_s}\equiv  u \pmod p}} \rho_1(x_1) \ldots  \rho_s(x_s),\\
\eta_2(v)  & =  \sum_{\substack{(x_{s+1}, \ldots, x_n) \in \cB_2\\ x_{s+1}^{e_{s+1}}\ldots x_n^{e_n}\equiv  v \pmod p}}
\rho_{s+1}(x_{s+1}) \ldots  \rho_n(x_n).
\end{split}
\end{equation*}
Therefore
\begin{equation}
\label{eq:S eta}
S_{p}(\lambda,\brho,\be;\cB) = \sum_{u, v=1}^{p-1} \eta_1(u) \eta_2(v) \ep(\lambda u v). 
\end{equation}
Recalling~\eqref{eq:rho} and using Lemmas~\ref{lem:GProdSet} and~\ref{lem:Kss}, we obtain
\begin{equation*}
\begin{split}
\sum_{u =1}^{p-1} \eta_1(u)^2 & \le\(h^{s} +  h^{2s} p^{-s/d_s} \) h^{o(1)}, \\
\sum_{v=1}^{p-1} \eta_2(v)^2 & \le\(h^{t} +  h^{2t} p^{-t/d_t} \) h^{o(1)},
\end{split} 
\end{equation*} 
where the integers $d_\nu$ are defined in Section~\ref{sec:cong}.

We now use the standard method of estimating bilinear sums via the Cauchy 
inequality, see, for example,~\cite[Lemma~4.1]{Gar} and derive  from~\eqref{eq:S eta} that
\begin{equation}
\label{eq:prelim}
\begin{split}
|S_{p}&(\lambda,\brho,\be;\cB)| \le \(p \sum_{u =1}^{p-1} \eta_1(u)^2   \sum_{v=1}^{p-1} \eta_2(v)^2\)^{1/2} \\
  & \le  p^{1/2}  \(h^s +  h^{2s} p^{-s/d_s}\)^{1/2}  \(h^t + h^{2t} p^{-t/d_t} \)^{1/2} h^{o(1)}\\
 & = p^{1/2} \(h^n + h^{n+s}p^{-s/d_s} +  h^{n+t}p^{-t/d_t}
  + h^{2n}p^{-s/d_s-t/d_t}\)^{1/2}  h^{o(1)}.
\end{split}
\end{equation}
Note that for $n \le 7$ we have $s \le t \le 4$  
and also 
$$
\frac{2}{d_2} =1 > \frac{3}{d_3} = \frac{3}{4}  > \frac{4}{d_4} = \frac{1}{2}.
$$
Thus for $n \le 7$ we have $h^{s}p^{-s/d_s}\le  h^{t}p^{-t/d_t}$ 
and the  bound~\eqref{eq:prelim} simplifies as
$$
|S_{p}(\lambda,\brho,\be;\cB)| \le  
p^{1/2} \(h^n +   h^{n+t}p^{-t/d_t}+ h^{2n}p^{-s/d_s-t/d_t}\)^{1/2}  h^{o(1)}.
$$

Now, for $n = 4$ we obtain 
\begin{equation}
\label{eq:n4}
\begin{split}
|S_{p}(\lambda,\brho&,\be;\cB)|  \le  
p^{1/2} \(h^4 +   h^{6}p^{-1}+ h^{8}p^{-2}\)^{1/2}  h^{o(1)}\\
&  \le  
p^{1/2} \(h^4 +     h^{8}p^{-2}\)^{1/2}  h^{o(1)} = 
h^{2+o(1)} p^{1/2} +     h^{4+o(1)}p^{-1/2}, 
\end{split}
\end{equation}
as the middle term never dominates.

Similarly, for $n = 5$, we obtain 
\begin{equation}
\label{eq:n5}
\begin{split}
|S_{p}(\lambda,\brho,\be;\cB)| & \le  
p^{1/2} \(h^5 +   h^{8}p^{-3/4}+ h^{10}p^{-7/4}\)^{1/2}  h^{o(1)}\\
&  \le  
h^{5/2+o(1)} p^{1/2} +     h^{4+o(1)}p^{1/8} + h^{5+o(1)}p^{-3/8} .
\end{split}
\end{equation}

For $n=6$, we see again that the middle term never dominates, so
\begin{equation}
\label{eq:n6}
\begin{split}
|S_{p}(\lambda,\brho,\be;\cB)| & \le  
p^{1/2} \(h^6 +    h^{12}p^{-3/2}\)^{1/2}  h^{o(1)}\\
&  = 
h^{3+o(1)} p^{1/2} +   h^{6+o(1)}p^{-1/4} .
\end{split}
\end{equation}

Finally, for $n=7$, we derive
\begin{equation*}
\begin{split}
|S_{p}(\lambda,\brho,\be;\cB)| & \le  
p^{1/2} \(h^7 +   h^{11}p^{-1/2}+ h^{14}p^{-5/4}\)^{1/2}  h^{o(1)}\\
&  \le  
h^{7/2+o(1)} p^{1/2} +     h^{11/2+o(1)}p^{1/4} + h^{7+o(1)}p^{-1/8} .
\end{split}
\end{equation*}
Note that for $h\le p^{1/6}$ we have $h^{11/2}p^{1/4}\le h^7$, thus the above
bound is trivial. On the other hand, for $h> p^{1/6}$ we have
$h^{7/2} p^{1/2}<h^{11/2}p^{1/4}$. Hence, we derive
\begin{equation}
\label{eq:n7}
|S_{p}(\lambda,\brho,\be;\cB)|  \le h^{11/2+o(1)}p^{1/4} + h^{7+o(1)}p^{-1/8} .
\end{equation}

Collecting the bounds~\eqref{eq:n4}, \eqref{eq:n5}, \eqref{eq:n6}
and~\eqref{eq:n7} we obtain the desired result. 
\end{proof}

It is easy to see that Theorem~\ref{thm:Kloost}  is nontrivial 
for any $\varepsilon>0$ and 
\begin{equation}
\label{eq:triv}
h \ge 
 \left\{  \begin{array}{ll}
p^{1/4+\varepsilon}  & \quad \text{if}\  n=4, \\
p^{1/5+\varepsilon} & \quad \text{if}\  n=5, \\
p^{1/6+\varepsilon} & \quad \text{if}\ n=6,\\
p^{1/6+\varepsilon} & \quad \text{if}\ n=7.
\end{array} \right.
\end{equation}

It is easy to see that a full analogue of Theorem~\ref{thm:Kloost} 
also holds for the sums $T_{p}(\lambda, \chi, \brho,\be;\cB)$.  

\begin{theorem}
\label{thm:Char-1}  
For  any  prime  $p$,   integer  $h \le p$, cube $\cB$ of the form~\eqref{eq:Box},
weights   $\brho$ with~\eqref{eq:rho}, vector  $\be$  with~\eqref{eq:ej}, 
nontrivial multiplicative characters $\chi$ modulo $p$,
and integer $\lambda$ with $\gcd(\lambda,p)=1$, 
we have
$$
|T_{p}(\lambda, \chi, \brho,\be;\cB)|  \le  h^{o(1)}\cdot
 \left\{  \begin{array}{ll}
h^{2} p^{1/2} +     h^{4}p^{-1/2} & \quad \text{if}\  n=4, \\
h^{5/2} p^{1/2} +     h^{4}p^{1/8} + h^{5}p^{-3/8} & \quad \text{if}\  n=5, \\
h^{3} p^{1/2} +     h^{6}p^{-1/4}& \quad \text{if}\ n=6,\\
 h^{11/2}p^{1/4} + h^{7}p^{-1/8}
& \quad \text{if}\ n=7.
\end{array} \right.
$$
\end{theorem}

We now show that for sums of multiplicative characters one can 
derive yet another estimate.

\begin{theorem}
\label{thm:Char-2}  
For  any  prime  $p$,   integer  $h \le p$, cube $\cB$ of the form~\eqref{eq:Box},
weights   $\brho$ with~\eqref{eq:rho}, vector  $\be$  with~\eqref{eq:ej}, 
nontrivial multiplicative characters $\chi$ modulo $p$,
and integer $\lambda$ with $\gcd(\lambda,p)=1$, 
\begin{itemize}
\item if $n=3$ then 
$$
|T_{p}(\lambda, \chi, \brho,\be;\cB)|  \le h^{o(1)}\cdot
 \left\{  \begin{array}{ll}
 h^{5/2} & \ \text{if}\  p> h \ge p^{1/2}, \\
h^{3/2}p^{1/2} & \ \text{if}\  p^{1/2} > h \ge p^{3/8}, \\
h^{5/2}p^{1/8} & \ \text{if}\  p^{3/8}> h \ge p^{1/4},
\end{array} \right.
$$

\item if $n=4$ then 
$$
|T_{p}(\lambda, \chi, \brho,\be;\cB)|  \le h^{o(1)}\cdot
 \left\{  \begin{array}{ll}
  h^{4}p^{-1/4} & \ \text{if}\  p> h \ge p^{1/2}, \\
 h^{2}p^{1/2} & \ \text{if}\  p^{1/2}> h \ge p^{9/32}, \\
h^{4}p^{-1/16} & \ \text{if}\  p^{9/32} > h \ge p^{1/4}, \\
h^{11/4}p^{1/4} & \ \text{if}\  p^{1/4}> h \ge p^{2/9},\\
h^{7/2}p^{1/12} & \ \text{if}\  p^{2/9}> h \ge p^{1/6}.
\end{array} \right.
$$
\end{itemize}
\end{theorem}

\begin{proof}  
We define 
$$
\cB_0   =  [k_1+1, k_1+h]\times \ldots \times[k_{n-1}+1, k_{n-1}+h] 
$$
and write 
$$
\eta_0(u)  =  \sum_{\substack{(x_1, \ldots, x_{n-1}) \in \cB_0\\ 
x_1^{e_1}\ldots x_{n-1}^{e_{n-1}}\equiv  u \pmod p}} \rho_1(x_1) \ldots  \rho_{n-1}(x_{n-1}).
$$
Therefore
\begin{equation*}
\begin{split}
T_{p}(\lambda, \chi, \brho,\be;\cB) &= \sum_{u=1}^{p-1} \eta_0(u) \sum_{x=k_n+1}^{k_n +h}\rho_n(x_n)\chi (ux +\lambda )\\
&= \sum_{u=1}^{p-1} \(\eta_0(u)^2\)^{1/2r}  \(\eta_0(u)\)^{(r-1)/r}\sum_{x=k_n+1}^{k_n +h}\rho_n(x)\chi (ux +\lambda ) . 
\end{split} 
\end{equation*}
Thus, by the H{\"o}lder inequality 
\begin{equation}
\begin{split}
\label{eq:T prelim}
|T_{p}(\lambda, \chi, \brho,\be;\cB)|^{2r}   = 
\sum_{u=1}^{p-1} \eta_0(u)^2  &\(\sum_{u=1}^{p-1} \eta_0(u)\)^{2r-2} \\
& \sum_{u=1}^{p-1} \left|\sum_{x=k_n+1}^{k_n +h}\rho_n(x_n)\chi (ux +\lambda )\right|^{2r}.
\end{split} 
\end{equation}
Clearly
$$
\sum_{u=1}^{p-1} \eta_0(u) = h^{n-1}.
$$
Recalling~\eqref{eq:rho} and using Lemmas~\ref{lem:GProdSet} and~\ref{lem:Kss}, we obtain
$$
\sum_{u =1}^{p-1} \eta_0(u)  \le\(h^{n-1} +  h^{2n-2} p^{-(n-1)/d_{n-1}} \) h^{o(1)}.
$$
Substituting the above bounds in~\eqref{eq:T prelim}, and using Lemma~\ref{lem:DavErd}
with $r=1$ we obtain
\begin{equation*}
\begin{split}
|T_{p}(\lambda, \chi, \brho,\be;\cB)|^{2} & \le 
\(h^{n-1} +  h^{2n-2} p^{-(n-1)/d_{n-1}} \) h^{1+o(1)}p \\
& \le \(h^{n}p +  h^{2n-1}p^{1-(n-1)/d_{n-1}}\) h^{o(1)}.
\end{split} 
\end{equation*}
Similarly, we derive from Lemma~\ref{lem:DavErd} that for $r =2,3,\ldots$
we have  
\begin{equation*}
\begin{split}
|T_{p}(\lambda, \chi&, \brho,\be;\cB)|^{2r} \le 
\(h^{n-1} +  h^{2n-2} p^{-(n-1)/d_{n-1}} \) h^{(n-1)(2r-2)+o(1)} \\
& \qquad\qquad\qquad\qquad\qquad\qquad\qquad \qquad\qquad \qquad \(h^r p +  h^{2r}p^{1/2}\) \\
& = \( h^{(n-1)(2r-1)}  +  h^{2(n-1)r} p^{-(n-1)/d_{n-1}} \)  
 \(h^r p +  h^{2r}p^{1/2}\) h^{o(1)}\\
 & = (  h^{2nr -n - r +1} p +  h^{(2n-1)r} p^{1-(n-1)/d_{n-1}}  
 \\
 & \qquad\qquad \qquad\qquad+  h^{2nr -n +1}  p^{1/2} 
 +h^{2nr} p^{1/2 -(n-1)/d_{n-1}})  h^{o(1)}.
\end{split} 
\end{equation*}

Now, for $n = 3$ we obtain 
\begin{equation}
\label{eq:n3char r1}
\begin{split}
|T_{p}(\lambda, \chi, \brho,\be;\cB)|^2 \le    
\(h^{3}p +  h^{5}\) h^{o(1)}
\end{split} 
\end{equation}
if $r=1$, while for $r =2,3,\ldots$ we have
\begin{equation}
\label{eq:n3char r2}
\begin{split}
|T_{p}(\lambda, \chi, \brho,\be;\cB)|^{2r}     
 & \le \(  h^{5r -2} p +  h^{5r} +  h^{6r -2}  p^{1/2} 
 +h^{6r} p^{-1/2}\) h^{o(1)}\\
  & \le \(  h^{5r -2} p +  h^{6r -2}  p^{1/2} 
 +h^{6r} p^{-1/2}\) h^{o(1)}, 
\end{split} 
\end{equation}
as the second term never dominates.
We now use~\eqref{eq:n3char r1} for $h \ge p^{3/8}$ and 
use~\eqref{eq:n3char r2} with $r = 2$ for  $h \ge p^{3/8}$,
getting the desired result for $n = 3$.

Similarly, for $n = 4$ we obtain 
\begin{equation}
\label{eq:n4char r1}
|T_{p}(\lambda, \chi, \brho,\be;\cB)|^{2}     
\le  \(h^4p +  h^{7}p^{1/4}\) h^{o(1)}
 h^{o(1)} 
\end{equation}
if $r=1$, while for $r =2,3,\ldots$ we obtain 
\begin{equation}
\label{eq:n4char r2}
|T_{p}(\lambda, \chi, \brho,\be;\cB)|^{2r}     
\le \(  h^{7r -3} p +  h^{7r} p^{1/4}  +  h^{8r - 3}  p^{1/2} 
 +h^{8r} p^{-1/4}\) h^{o(1)}. 
\end{equation}
We now notice that~\eqref{eq:n4char r1} is always weaker that the bound
of Theorem~\ref{thm:Char-1}  (for $n=4$), 
which we use for $h \ge p^{9/32}$.
We also use~\eqref{eq:n4char r2} with $r = 3$ for  $h< p^{2/9}$
we also use~\eqref{eq:n4char r2} with $r = 2$ for  
$p^{2/9} \le h < p^{9/32}$,
getting the desired result for $n = 4$. 
\end{proof}

Clearly,  Theorem~\ref{thm:Char-1}  is nontrivial  under the
condition~\eqref{eq:triv}, while  Theorem~\ref{thm:Char-2} 
is nontrivial for 
$$
h \ge 
 \left\{  \begin{array}{ll}
p^{1/4+\varepsilon}  & \quad \text{if}\  n=3,  \\
p^{1/6+\varepsilon} & \quad \text{if}\ n=4.
\end{array} \right.
$$

\section{Bounds for almost all primes}

\begin{theorem}
\label{thm:Kloost AA}  
For a sufficiently large positive integer  $T$, $h\ge3$,
for all but $o(T/\log T)$ primes $p \le T$, uniformly 
over positive integers $h \le p$, cubes $\cB$ of the form~\eqref{eq:Box},
weights   $\brho$ with~\eqref{eq:rho}, vectors $\be$  with~\eqref{eq:ej}
and integers $\lambda$ with $\gcd(\lambda,p)=1$, 
we have
$$
|S_{p}(\lambda,\brho,\be;\cB)| \le   \(h^{n/2} p^{1/2} + h^{n/2+\rf{n/2}/2-1/4}p^{1/4}  + h^{n -1/2} \)  h^{o(1)}
$$
\end{theorem}

\begin{proof} Let $p$ be one of the non-exceptional primes for which 
the bound of  Lemma~\ref{lem:GProdSet AlmostAll} holds.

We set $s = \fl{n/2}$ and $t = n-s = \rf{n}$
and define $\eta_1(u)$ and $\eta_2(v)$ as in the proof of Theorem~\ref{thm:Kloost}. 
Recalling~\eqref{eq:rho} and using Lemmas~\ref{lem:GProdSet AlmostAll} and~\ref{lem:Kss}, we obtain
\begin{equation*}
\begin{split}
\sum_{u =1}^{p-1} \eta_1(u)^2 & \le \(h^s + h^{2s-1/2}p^{-1/2}\)h^{o(1)}, \\
\sum_{v=1}^{p-1} \eta_2(v)^2 & \le \(h^t + h^{2t-1/2}p^{-1/2}\)h^{o(1)}.
\end{split} 
\end{equation*}
We now apply the standard method of estimating bilinear sums via the Cauchy 
inequality, see, for example,~\cite[Lemma~4.1]{Gar} and derive from~\eqref{eq:S eta}
that 
\begin{equation*}
\begin{split}
|S_{p}(\lambda,\brho&,\be;\cB)| \le \(p \sum_{u =1}^{p-1} \eta_1(u)^2   \sum_{v=1}^{p-1} \eta_2(v)^2\)^{1/2} \\
  \le & p^{1/2}  \(h^s + h^{2s-1/2}p^{-1/2}\)^{1/2}  \(h^t + h^{2t-1/2}p^{-1/2}\)^{1/2} h^{o(1)}\\
 = & p^{1/2} \(h^n + h^{n+t-1/2}p^{-1/2} + h^{n+s-1/2}p^{-1/2} + h^{2n -1 }p^{-1}\)^{1/2}  h^{o(1)}\\
  = & p^{1/2} \(h^n + h^{n+t-1/2}p^{-1/2}  + h^{2n -1}p^{-1}\)^{1/2}  h^{o(1)}, 
\end{split}
\end{equation*}
and the result now follows. 
\end{proof}

It is easy to see that for an even $n$ the middle term in 
the bound of Theorem~\ref{thm:Kloost AA}  
never dominates and the bound simplifies as 
$$
|S_{p}(\lambda,\brho,\be;\cB)| \le   \(h^{n/2} p^{1/2} + h^{n -1/2} \)  h^{o(1)}.
$$

Clearly,  Theorem~\ref{thm:Kloost AA}  is nontrivial
provided that $h \ge p^{1/n + \varepsilon}$ for some fixed $\varepsilon > 0$.

Furthermore, as before, we also have an analogue of Theorem~\ref{thm:Kloost AA}
for the sums of multiplicative characters. 

\begin{theorem}
\label{thm:Char-1 AA}  
For a sufficiently large positive integer  $T$, $h\ge3$,
for all but $o(T/\log T)$ primes $p \le T$, uniformly 
over positive integers $h \le p$, cubes $\cB$ of the form~\eqref{eq:Box},
weights   $\brho$ with~\eqref{eq:rho}, vectors $\be$  with~\eqref{eq:ej}, 
nontrivial multiplicative characters $\chi$ modulo $p$,
and integers $\lambda$ with $\gcd(\lambda,p)=1$, 
we have
$$
|T_{p}(\lambda, \chi, \brho,\be;\cB)| \le   \(h^{n/2} p^{1/2} + h^{n/2+\rf{n/2}/2-1/4}p^{1/4}  + h^{n -1/2} \)  h^{o(1)}. 
$$
\end{theorem}

Finally, using Lemma~\ref{lem:GProdSet AlmostAll} instead of 
Lemma~\ref{lem:GProdSet} in the proof of Theorem~\ref{thm:Char-2}, 
we obtain:

\begin{theorem}
\label{thm:Char-2 AA}  
For a sufficiently large positive integer  $T$, $h\ge3$,
for all but $o(T/\log T)$ primes $p \le T$, uniformly 
over positive integers $h \le p$, cubes $\cB$ of the form~\eqref{eq:Box},
weights   $\brho$ with~\eqref{eq:rho}, vectors $\be$  with~\eqref{eq:ej},
nontrivial multiplicative characters $\chi$ modulo $p$,
and integers $\lambda$ with $\gcd(\lambda,p)=1$, 
we have
\begin{equation*}
\begin{split}
|T_{p}(\lambda, \chi, \brho,\be;\cB)|    
&\le 
(h^{n-1/2 -(n-1)/2r} p^{1/2r} +  h^{n-1/2-1/4r}  p^{1/4r} \\
& \qquad \qquad \qquad \qquad + h^{n -(n-1)/2r}    p^{1/4r}
 +h^{n -1/4r})h^{o(1)}  \end{split} 
\end{equation*}
for  $r=2,3 \ldots$. 
\end{theorem}

Clearly one can also obtain a version of Theorem~\ref{thm:Char-2 AA}  with 
$r=1$, that is, 
$$
|T_{p}(\lambda, \chi, \brho,\be;\cB)|  \le 
(h^{n/2} p^{1/2} +  h^{n-3/4}  p^{1/4} )h^{o(1)}.  
$$
However it is always weaker than the bound of Theorem~\ref{thm:Char-1 AA}.

One easily verifies that Theorems~\ref{thm:Char-1 AA} and~\ref{thm:Char-2 AA}  are nontrivial
provided that $h \ge p^{1/n + \varepsilon}$ and $h \ge p^{1/2(n-1) + \varepsilon}$, 
respectively, for some fixed $\varepsilon > 0$.

\section{Comments}
\label{sec:com}

It is easy to see that the implied constants depend only on $\gcd(p-1,e_i)$
rather on $e_i$, $i =1, \ldots, n$. Thus our results remain nontrivial
even for very large values of $e_1,\ldots, e_n$, provided that the corresponding 
greatest common divisors are small. 

Furthermore, the method of this work can certainly be adjusted to 
apply to sums over more general boxes
$$
\cB = [k_1+1, k_1+h_1]\times  \ldots \times[k_n+1, k_n+h_n]
$$
with distinct sides. However finding an optimal choice of parameters
can be quite technically cluttered in this case. 

Clear the bounds of Theorems~\ref{thm:Kloost} and~\ref{thm:Char-1}  
and  of Theorem~\ref{thm:Char-2}  also apply to sums with 
$n \ge 8$ and $n \ge 5$ variables, respectively.  Simply, for each choice
of ``unused'' variables, we use one of these bounds. 
However, it is natural to expect that using more variables may
lead to stronger bounds over smaller cubes $\cB$. 
For example, one can try to split all variables into three groups 
and the replace the estimate~\eqref{eq:prelim} with an appropriate 
variant of the bound of
Bourgain and Garaev~\cite[Theorem~1.2]{BouGar1} of trilinear sums.
More precisely, one needs a bound on trilinear sums with weights 
bounded in $L_2$-norm rather than in $L_\infty$-norm as in~\cite{BouGar1},
see also~\cite{Bour1,Bour2}. 

\section*{Acknowledgements}

This work was finished during a very enjoyable 
stay of the author at the Max Planck Institute for Mathematics,
Bonn. 
It was also
supported in part by ARC grants DP110100628 and
DP130100237

%{\small


\begin{thebibliography}{9999}

\bibitem{ACZ} A. Ayyad, T. Cochrane and Z. Zheng, 
`The congruence $x_1x_2 \equiv x_3x_4 \pmod p$, the equation
$x_1x_2 = x_3x_4$ and the mean value of character sums', 
{\it J. Number Theory\/}, {\bf 59} (1996), 398--413.


\bibitem{Bour1} J. Bourgain, 
`Multilinear exponential sums in prime
fields under optimal entropy condition on
the sources', {\it  Geom. and Func. Anal.\/}, 
{\bf 18} (2009), 1477--1502.
 

\bibitem{Bour2} J.~Bourgain,  
`On exponential sums in finite fields',
{\it  An Irregular Mind\/}, 
Bolyai Soc. Math. Stud., vol.~21, J{\'a}nos Bolyai Math. Soc., 
Budapest, 2010, 219--242.

\bibitem{BouCha}
J. Bourgain and M.-C. Chang, `On a multilinear  character sums of Burgess', 
{\it  Comp. Rend. Acad. Sci. Paris\/}, {\bf 348}
(2010),  115--120.


\bibitem{BouGar1}
J. Bourgain and M. Z. Garaev, `On a variant of sum-product estimates
and explicit exponential sum bounds in prime fields', {\it Math.
Proc. Cambr. Phil. Soc.\/}, {\bf 146} (2008), 1--21.

\bibitem{BouGar2}
J. Bourgain and M. Z. Garaev, `Sumsets of reciprocals in prime fields 
and multilinear Kloosterman sums', {\it Preprint\/}, 2012, 
(available from {\tt http://arxiv.org/abs/1211.4184}).


\bibitem{BGKS2} J.~Bourgain, M.~Z.~Garaev, S. V. Konyagin
and I. E. Shparlinski, `On congruences with products of variables
from short intervals and  applications',
 {\it Proc. Steklov Math. Inst.\/}, {\bf 280} (2013),  67--96. 

\bibitem{BGKS3} J.~Bourgain, M.~Z.~Garaev, S. V. Konyagin 
and I. E. Shparlinski, `Multiplicative  congruences with 
variables from short intervals', 
{\it  J. d'Analyse Math.\/}, (to appear). 

\bibitem{DavErd}
H. Davenport and P. Erd{\H o}s, `The distribution of quadratic and higher residues', 
{\it PubL Math. Debrecen\/}, {\bf 2} (1952), 252--265.

\bibitem{Del} P. Deligne, `Applications de la formule des traces aux
sommes trigonom\'etriques', {\it Lect. Notes in Mathematics\/},
Springer-Verlag, Berlin, {\bf 569} (1977), 168--232.


\bibitem{Gar} M. Z. Garaev, `Sums and products of sets and estimates of rational
trigonometric sums in fields of prime order', 
{\it Russian Math. Surveys\/}, {\bf  65}  (2010),  599--658 
(Transl. from {\it Uspekhi Mat. Nauk\/}).


 \bibitem{IwKow} H. Iwaniec and E. Kowalski,
{\it Analytic number theory\/}, Amer.  Math.  Soc.,
Providence, RI, 2004.


%\bibitem{Kar1} A. A. Karatsuba, `The distribution of values
%of  Dirichlet characters on additive sequences', {\it Doklady
%Acad. Sci. USSR\/}, {\bf 319} (1991),   543--545 (in   Russian).
%
%\bibitem{Kar2} A. A. Karatsuba, {\it Basic analytic number theory\/}, 
%Springer-Verlag, 1993.


\bibitem{Luo} W. Luo, `Bounds for incomplete hyper-Kloosterman sums', 
{\it J. Number Theory\/}, \textbf{75} (1999), 41--46.

\bibitem{Shp} I. E. Shparlinski, `Bounds of incomplete multiple Kloosterman sums',  
{\it J. Number Theory\/}, {\bf 126} (2007), 68--73.

\bibitem{WaLi} Y. Wang and H.~Li,
`On $s$-dimensional incomplete Kloosterman sums',  
{\it J. Number Theory\/}, {\bf 130} (2010), 1602--1608.

\bibitem{Wool} T.~D.~Wooley,
`Vinogradov's mean value theorem via efficient congruencing, II',
{\it Duke Math. J.\/}  (to appear).

\end{thebibliography}
\end{document}